\documentclass[12pt]{amsart}
\usepackage{amsmath}
\usepackage{enumerate}
\usepackage{amssymb}
\usepackage{amsbsy}
\usepackage{amsfonts}
\theoremstyle{plain}
\newtheorem{thm}{Theorem}
\newtheorem{prop}{Proposition}[section]
\newtheorem{lem}[prop]{Lemma}
\newtheorem{thr}[prop]{Theorem}

\newtheorem{cor}[thm]{Corollary}
\newtheorem*{con}{Newman's Conjecture}
\theoremstyle{definition}

\newtheorem{rem}[prop]{Remark}
\newtheorem{dfn}[prop]{Definition}

\newtheorem*{re*}{Remark}
\newtheorem*{ex*}{Example}
\newtheorem*{pro*}{Property A}
\numberwithin{equation}{section}

\newcommand{\sm}{\left(\begin{smallmatrix}}
\newcommand{\esm}{\end{smallmatrix}\right)}

\newfont{\FieldFont}{msbm10 scaled\magstep1}

\begin{document}

\title[Distribution of integral Fourier Coefficients
Modulo Primes] {Distribution of integral Fourier Coefficients of a
Modular Form of Half Integral Weight Modulo Primes}

\author{D. choi }
\address{School of Mathematics, KIAS, 207-43 Cheongnyangni 2-dong 130-722, Korea}
\email{choija@postech.ac.kr}


\begin{abstract}
Recently, Bruinier and Ono classified cusp forms
$f(z):=\sum_{n=0}^{\infty}a_f(n)q^n\in
S_{\lambda+\frac{1}{2}}(\Gamma_0(N),\chi)\cap \mathbb{Z}[[q]]$
that does not satisfy a certain distribution property for modulo
odd primes $p$. In this paper, using Rankin-Cohen Bracket, we
extend this result to modular forms of half integral weight for
primes $p \geq 5$. As applications of our main theorem we derive
distribution properties, for  modulo primes $p\geq5$, of traces of
singular moduli and Hurwitz class number. We also study an
analogue of Newman's conjecture for overpartitions.
\end{abstract}

 \maketitle


\section{Introduction and Results}
Let $M_{\lambda+\frac{1}{2}}(\Gamma_0(N),\chi)$ and
$S_{\lambda+\frac{1}{2}}(\Gamma_0(N),\chi)$ be the spaces,
respectively, of modular forms and cusp forms of weight
$\lambda+\frac{1}{2}$ on $\Gamma_0(N)$ with a Dirichlet character
$\chi$ whose conductor divides $N$. If $f(z)\in
M_{\lambda+\frac{1}{2}}(\Gamma_0(N),\chi)$, then $f(z)$ has the
form
$$f(z)=\sum_{n=0}^{\infty}a(n)q^n,$$ where $q:=e^{2 \pi i z}$.
It is well-known that the coefficients of $f$ are related to
interesting objects in number theory such as the special values of
$L$-function, class number, traces of singular moduli and so on.
In this paper, we study congruence properties of the Fourier
coefficients of $f(z)\in
M_{\lambda+\frac{1}{2}}(\Gamma_0(N),\chi)\cap \mathbb{Z}[[q]]$ and
their applications.

Recently, Bruinier and Ono proved in \cite{B-O} that $g(z)\in
S_{\lambda+\frac{1}{2}}(\Gamma_0(N),\chi)\cap \mathbb{Z}[[q]]$ has
a special form (see (\ref{2.1})) by modulo $p$ when $p$ is an odd
prime and the coefficients of $f(z)$ do not satisfy the following
property for $p$:
\begin{pro*}
If $M$ is a positive integer, we say that a sequence $\alpha(n)
\in \mathbb{Z}$ satisfies Property A for $M$ if for every integer
$r$
$$\begin{array}{c}
  \sharp\{\; 1 \leq n \leq X \; | \; \alpha(n) \equiv r \pmod{M} \text{
and } \gcd(M,n)=1\}\\
  \gg_{r,M}\left\{\begin{array}{cl}
  \frac{\sqrt{X}}{logX} \;\;\;& \text{ if } r \not \equiv 0 \pmod{M}, \\
  X \quad \;\;&\text{ if } r \equiv 0 \pmod{M}.
\end{array}\right.
\end{array}
$$
\end{pro*}
Let
$$\theta(f(z)):=\frac{1}{2 \pi i}\cdot\frac{d}{dz}f(z)=\sum_{n=1}^{\infty}n\cdot a(n)q^n.$$
Using Rankin-Cohen Bracket (see (\ref{cohen})), we prove that
there exists $$\widetilde{f}(z) \in
S_{\lambda+p+1+\frac{1}{2}}(\Gamma_0(4N),\chi)\cap
\mathbb{Z}[[q]]$$ such that $\theta(f(z))\equiv \widetilde{f}(z)
\pmod{p}$. We extend the results in \cite{B-O} to modular forms of
half integral weight.
\begin{thm}\label{main1}
Let $\lambda$ be a non-negative integer. We assume that
$f(z)=\sum_{n=0}^{\infty}a(n)q^n\in
M_{\lambda+\frac{1}{2}}(\Gamma_0(4N),\chi)\cap \mathbb{Z}[[q]],$
where $\chi$ is a real Dirichlet character. If $p\geq 5$ is a
prime and there exists a positive integer $n$ for which
$\gcd(a(n),p)=1$ and $\gcd(n,p)=1$, then at least one of the
following is true:
\begin{enumerate}
\item The coefficients of $\theta^{p-1}(f(z))$ satisfies Property
A for $p$. \item There are finitely many square-free integers
$n_1$, $n_2, \cdots,n_t$ for which
\begin{equation}\label{form}
\theta^{p-1}(f(z))\equiv \sum_{i=1}^t\sum_{m=0}^{\infty}a(n_i
m^2)q^{n_i m^2} \pmod{p}.\end{equation} Moreover, if
$\gcd(4N,p)=1$ and an odd prime $\ell$ divides some $n_i$, then
$$p|(\ell-1)\ell(\ell+1)N \text{ or } \ell\mid N.$$
\end{enumerate}
\end{thm}

\begin{rem}
Note that for every odd prime $p\geq5$,
$$\theta^{p-1}(f(z))\equiv \sum_{\begin{smallmatrix}
  n>0 \\
  p \nmid {n}
\end{smallmatrix}}a(n)q^n \pmod{p}. $$
\end{rem}

As an applications of Theorem \ref{main1}, we study the
distribution of traces of singular moduli modulo primes $p\geq5$.
Let $j(z)$ be the usual $j$-invariant function. We denote by $F_d$
the set of positive definite binary quadratic forms
$$F(x,y)=ax^2+bxy+cy^2=[a,b,c]$$ with discriminant $-d=b^2-4ac$.
For each $F(x,y)$, let $\alpha_F$ be the unique complex number in
the complex upper half plane, which is a root of $F(x,1)$. We
define $\omega_F \in \{1,2,3\}$ as
$$\omega_{F}:=\left\{\begin{array}{l}
  2 \;\;\text{ if } F\sim_{\Gamma}[a,0,a], \\
  3 \;\;\text{ if } F\sim_{\Gamma}[a,a,a], \\
  1 \;\;\text{ otherwise, }
\end{array}\right.
$$
where $\Gamma:=SL_2(\mathbb{Z})$. Here, $F\sim_{\Gamma}[a,b,c]$
denotes that $F(x,y)$ is equivalent to $[a,b,c]$. From these
notations, we define the Hecke trace of singular moduli.
\begin{dfn}
If $m \geq 1$, then we define the {\it mth Hecke trace of the
singular moduli of discriminant $-d$} as
$$t_m(d):=\sum_{F\in F_d/\Gamma}\frac{j_m(\alpha_F)}{\omega_{F}},$$
where $F_d/\Gamma$ denotes a set of $\Gamma-$equivalence classes
of $F_d$ and
$$j_m(z):=j(z)|T_0(m)=\sum_{
\begin{smallmatrix}
  d|m \\
  ad=m \\
\end{smallmatrix}}\sum_{b=0}^{d-1}j\left(\frac{az+b}{d}\right).$$
Here, $T_0(m)$ denotes the normalized $m$th weight zero Hecke
operator.
\end{dfn}
Note that $t_1(d)=t(d)$, where
$$t(d):=\sum_{F\in F_d/\Gamma}\frac{j(\alpha_F)-744}{\omega_{F}}$$
is the usual trace of singular moduli. Let
$$h(z):=\frac{\eta(z)^2}{\eta(2z)}\cdot
\frac{E_4(4z)}{\eta(4z)^6}$$ and $B_m(1,d)$ denote the coefficient
of $q^d$ in $h(z)|T(m^2,1,\chi_0)$, where
$$E_4(z):=1+240\sum_{n=1}^{\infty}\sum_{d|n}d^3 q^n, \;\eta(z):=q^{\frac{1}{24}} \prod_{n=1}^{\infty}\left(1-q^n\right),$$
and $\chi_0$ is a trivial character.  Here, $T(m^2,\lambda,\chi)$
denotes the $m$th Hecke operator of weight $\lambda+\frac{1}{2}$
with a Dirichlet chracter $\chi$ (see VI. $\S$3. in \cite{Ko} or
(\ref{Hecke})). Zagier proved in \cite{Z} that for all $m$ and $d$
\begin{equation}\label{Za} t_m(d)=-B_m(1,d).
\end{equation}
Using these generating functions, Ahlgren and Ono studied the
divisibility properties of traces and Hecke traces of singular
moduli in terms of the factorization of primes in imaginary
quadratic fields (see \cite{A-O}). For example, they proved that a
positive proportion of the primes $\ell$ has the property that
$t_m(\ell^3n)\equiv 0 \pmod{p^s}$ for every positive integer $n$
coprime to $\ell$ such that $p$ is inert or ramified in
$\mathbb{Q}\left(\sqrt{-n\ell}\right)$. Here, $p$ is an odd prime,
and $s$ and $m$ are integers with $p \nmid m$. In the following
theorem, we give the distribution of traces and Hecke traces of
singular moduli modulo primes $p$. \newpage
\begin{thm}\label{main2} Suppose that $p\geq5$ is a prime  such
that  $p\equiv 2 \pmod{3}$.
\begin{enumerate}
\item  Then, for every integer $r$, $p \nmid r$,
$$\sharp\{\; 1 \leq n \leq X \; | \; t_{1}(n) \equiv r
\pmod{p}\}\gg_{r,p}
    \left\{\begin{array}{cl}
  \frac{\sqrt{X}}{logX} \;\;\;& \text{ if } r \not \equiv 0 \pmod{p} \\
  X \quad \;\;&\text{ if } r \equiv 0 \pmod{p}.
\end{array}\right.$$
\item Then, a positive proportion of the primes $\ell$ has the
property that
$$\sharp\{\; 1 \leq n \leq X \; | \; t_{\ell}(n) \equiv r
\pmod{p}\}\gg_{r,p}
    \left\{\begin{array}{cl}
  \frac{\sqrt{X}}{logX} \;\;\;& \text{ if } r \not \equiv 0 \pmod{p} \\
  X \quad \;\;&\text{ if } r \equiv 0 \pmod{p}.
\end{array}\right.$$
  for every integer $r$, $p \nmid r$.
\end{enumerate}
\end{thm}

As another application we study the distribution of Hurwitz class
number modulo primes $p\geq 5$. The Hurwitz class number $H(-N)$
is defined as follows: the class number of quadratic forms of the
discriminant $-N$ where each class $C$ is counted with
multiplicity $\frac{1}{Aut(C)}$. The following theorem gives the
distribution of Hurwitz class number modulo primes $p\geq 5$.
\begin{thm}\label{main3}
Suppose that $p \geq 5$ is a prime. Then, for every integer $r$
  $$
  \sharp\{\; 1 \leq n \leq X \; | \; H(n) \equiv r
\pmod{p}\}
  \gg_{r,p}\left\{\begin{array}{cl}
  \frac{\sqrt{X}}{logX} \;\;\;&\text{ if } r \not \equiv 0 \pmod{p}, \\
  X \quad \;\;&\text{ if } r \equiv 0 \pmod{p}.
\end{array}\right.
$$
\end{thm}

 We also use the main theorem to study an analogue of
Newman's conjecture for overpartitions. Newman's conjecture
concerns the distribution of the ordinary partition function
modulo primes $p$.
\begin{con}
Let $P(n)$ be an ordinary partition function. If $M$ is a positive
integer, then for every integer $r$ there are infinitely many
nonnegative integer $n$ for which $P(n)\equiv r \pmod{M}$.
\end{con}
\noindent This conjecture was already studied by many
mathematicians (see Chapter 5. in \cite{O}). The overpartition of
a natural number $n$ is a partition of $n$ in which the first
occurrence of a number may be overlined. Let $\bar{P}(n)$ be the
number of the overpartition of an integer $n$. As an analogue of
Newman's conjecture, the following theorem gives a distribution
property of $\bar{P}(n)$ modulo odd primes $p$.
\begin{thm}\label{main4} Suppose that $p \geq 5$ is a prime such that
$p\equiv 2 \pmod{3}$. Then, for every integer $r$,
  $$
  \sharp\{\; 1 \leq n \leq X \; | \; \bar{P}(n) \equiv r
\pmod{p}\}
  \gg_{r,p}\left\{\begin{array}{cl}
  \frac{\sqrt{X}}{logX} \;\;\;&\text{ if } r \not \equiv  0 \pmod{p}, \\
  X \quad \;\;&\text{ if } r \equiv 0 \pmod{p}.
\end{array}\right.
$$
\end{thm}
\begin{rem}
When $r\equiv 0 \pmod{p}$, Theorem \ref{main2}, \ref{main3} and
\ref{main4} were proved in \cite{A-O} and \cite{Tr}.
\end{rem}
Next sections are detailed proofs of theorems: Section 2 gives a
proof of Theorem \ref{main1}. In Section 3, we give the proofs of
Theorem \ref{main2}, \ref{main3}, and \ref{main4}.

\section{\bf Proof of Theorem \ref{main1}}
We begin by stating the following theorem proved in \cite{B-O}.
\begin{thr}[\cite{B-O}]\label{B-O}
Let $\lambda$ be a non-negative integer. Suppose that
$g(z)=\sum_{n=0}^{\infty}a_g(n)q^n\in
S_{\lambda+\frac{1}{2}}(\Gamma_0(4N),\chi)\cap \mathbb{Z}[[q]],$
where $\chi$ is a real Dirichlet character. If $p$ is an odd prime
and a positive integer $n$ exists for which $\gcd(a_g(n),p)=1$,
then at least one of the following is true:
\begin{enumerate}
\item If $0\leq r <p$, then
$$ \sharp\{\; 1 \leq n \leq X \; | \; a_g(n) \equiv r \pmod{p}\}
  \gg_{r,M}\left\{\begin{array}{cl}
  \frac{\sqrt{X}}{logX} \;\;\;& \text{ if } r \not \equiv 0 \pmod{p}, \\
  X \quad \;\;&\text{ if } r \equiv 0 \pmod{p}.
\end{array}\right.
$$
\item There are finitely many square-free integers $n_1$, $n_2,
\cdots,n_t$ for which
\begin{equation}\label{2.1}
g(z)\equiv \sum_{i=1}^t\sum_{m=0}^{\infty}a_g(n_i m^2)q^{n_i m^2}
\pmod{p}.\end{equation} Moreover, if $\gcd(p,4N)=1$, $\epsilon \in
\{\pm1\}$, and $\ell \nmid 4Np$ is a prime with
$\left(\frac{n_i}{\ell}\right) \in \{0, \epsilon \}$ for $1\leq i
\leq t$, then $(\ell-1)g(z)$ is an eigenform modulo $p$ of the
half-integral weight Hecke operator $T(\ell^2, \lambda,\chi)$. In
particular, we have
\begin{equation}\label{hecke}
(\ell-1)g(z)|T(\ell^2,\lambda,\chi)\equiv \epsilon
\chi(p)\left(\frac{(-1)^{\lambda}}{\ell}\right)\left
(\ell^{\lambda}+\ell^{\lambda-1}\right)(\ell-1)g(z) \pmod{p}.
\end{equation}
\end{enumerate}
\end{thr}
Recall that $f(z)=\sum_{n=0}^{\infty}a(n)q^n\in
M_{\lambda+\frac{1}{2}}(\Gamma_0(4N),\chi)\cap \mathbb{Z}[[q]]$.
Thus, to apply Theorem \ref{B-O}, we show that there exists a cusp
form $\widetilde{f}(z)$ such that $\widetilde{f}(z)\equiv
\theta^{p-1}(f(z)) \pmod{p}$ for a prime $p\geq5$.
\begin{lem}\label{l-main1}
Suppose that $p \geq 5$ is a prime and
$$f(z)=\sum_{n=0}^{\infty}a(n)q^n \in
M_{\lambda+\frac{1}{2}}(\Gamma_0(N),\chi)\cap\mathbb{Z}[[q]].$$
Then, there exists a cusp form $\widetilde{f}(z)\in
S_{\lambda+(p+1)(p-1)+\frac{1}{2}}(\Gamma_0(N),\chi)\cap\mathbb{Z}[[q]]$
such that $$\widetilde{f}(z)\equiv \theta^{p-1}(f(z)) \pmod{p}.$$
\end{lem}
\begin{proof}[Proof of Lemma \ref{l-main1}]

For $F(z)\in M_{\frac{k_1}{2}}(\Gamma_0(N),\chi_1)$ and $G(z)\in
 M_{\frac{k_2}{2}}(\Gamma_0(N),\chi_2)$, let
\begin{equation}\label{cohen}
[F(z),G(z)]_1:=\frac{k_2}{2}\theta (F(z))\cdot G(z)-\frac{k_1}{2}
F(z)\cdot \theta
 (G(z)).
\end{equation}
This operator is referred to as a Rankin-Cohen 1-bracket, and it
was proved in \cite{C} that $$[F(z),G(z)]_1 \in
 S_{\frac{k_1+k_2}{2}+2}(\Gamma_0(N),\chi_1\chi_2\chi'),$$ where
 $\chi'=1$ if $\frac{k_1}{2}$ and $\frac{k_2}{2} \in \mathbb{Z}$,
 $\chi'(d)=\left(\frac{-4}{d}\right)^{\frac{k_i}{2}}$ if $\frac{k_i}{2}\in
 \mathbb{Z}$ and $\frac{k_{3-i}}{2}\in \frac{1}{2}+\mathbb{Z}$, and
 $\chi'(d)=\left(\frac{-4}{d}\right)^{\frac{k_1+k_2}{2}}$ if $\frac{k_1}{2}$ and
  $\frac{k_{2}}{2}\in \frac{1}{2}+\mathbb{Z}$.

For even $k\geq 4$, let
$$E_{k}(z):=1-\frac{2k}{B_k}\sum_{n=1}^{\infty}\sum_{d|n}d^{k-1}q^n$$ be the usual
normalized Eisenstein series of weight $k$. Here, the number $B_k$
denotes the $k$th Bernoulli number. The function $E_k(z)$ is a
modular form of weight $k$ on $SL_2(\mathbb{Z})$, and
\begin{equation}\label{Eis}
E_{p-1}(z) \equiv 1 \pmod{p}
\end{equation}
(see \cite{L}). From (\ref{cohen}) and (\ref{Eis}), we have
$$[E_{p-1}(z),f(z)]_1\equiv \theta(f(z)) \pmod{p}$$ and
$[E_{p-1}(z),f(z)]_1 \in
S_{\lambda+p+1+\frac{1}{2}}(\Gamma_0(N),\chi)$. Repeating this
method $p-1$ times, we complete the proof.
\end{proof}
Using the following lemma, we can deal with the divisibility of
$a_g(n)$ for positive integers $n$, $p \nmid n$, where
$g(z)=\sum_{n=1}^{\infty}a_g(n)q^n \in
S_{\lambda+\frac{1}{2}}(\Gamma_0(N),\chi)\cap \mathbb{Z}[[q]]$.
\begin{lem}[see Chapter 3 in \cite{O}]\label{2-3} Suppose that
$g(z)=\sum_{n=1}^{\infty}a_g(n)q^n\in
S_{\lambda+\frac{1}{2}}(\Gamma_0(N),\chi)$ has coefficients in
$\mathcal{O}_K$, the algebraic integers of some number field $K$.
Furthermore, suppose that $\lambda \geq 1$ and that $\textbf{m}
\subset \mathcal{O}_K$ is an ideal norm $M$.
\begin{enumerate}
\item Then, a positive proportion of the primes $Q \equiv -1
\pmod{4MN}$ has the property that
$$g(z)|T(Q^2,\lambda,\chi)\equiv0 \pmod{\textbf{m}}.$$
\item Then a positive proportion of the primes $Q \equiv 1
\pmod{4MN}$ has the property that
$$g(z)|T(Q^2,\lambda,\chi)\equiv 2g(z) \pmod{\textbf{m}}.$$
\end{enumerate}
\end{lem}
\noindent We can now prove Theorem \ref{main1}.
\begin{proof}[Proof of Theorem \ref{main1}]
From Lemma \ref{l-main1}, there exists a cusp form
$$\widetilde{f}(z)\in
S_{\lambda+(p+1)(p-1)+\frac{1}{2}}(\Gamma_0(N),\chi)\cap\mathbb{Z}[[q]]$$
such that $$\widetilde{f}(z)\equiv \theta^{p-1}(f(z)) \pmod{p}.$$
Note that, for $F(z)=\sum_{n=0}^{\infty}a_F(n)q^n \in
M_{k+\frac{1}{2}}(\Gamma_0(N),\chi)$ and each prime $Q \nmid N$,
the half-integral weight Hecke operator $T(Q^2,\lambda,\chi)$ is
defined as
\begin{equation}\label{Hecke}
\begin{array}{l}
  F(z)|T(Q^2,k,\chi) \\
  \qquad \quad:=\sum_{n=0}^{\infty}\left(a_F(Q^2 n)+
\chi^*(Q)\left(\frac{n}{Q}\right)Q^{k-1}a_F(n)
+\chi^*(Q^2)Q^{2k-1}a_F\left(n/Q^2\right)\right)q^n,
\end{array}
\end{equation}
 where
$\chi^*(n):=\chi^*(n)\left(\frac{(-1)^{k}}{n}\right)$ and
$a_F(n/Q^2)=0$ if $Q^2 \nmid n$. If $F(z)|T(Q^2,k,\chi)\equiv0
\pmod{p}$ for a prime $Q \nmid N$, then we have \begin{align*}
a_F(Q^2\cdot Qn)&+
\chi^*(Q)\left(\frac{Qn}{Q}\right)Q^{k-1}a_F(Qn)
+\chi^*(Q^2)Q^{2k-1}a_F\left(Qn /Q^2\right)\\
&\equiv a_F(Q^3n) \equiv 0 \pmod{p}
\end{align*}
  for every positive
integer $n$ such that $\gcd(Q,n)=1$. Thus, we have the following
by Lemma \ref{2-3}-(1): $$\sharp\{\; 1 \leq n \leq X \; | \; a(n)
\equiv 0 \pmod{p} \text{ and } \gcd(p,n)=1\} \gg X.$$

We apply Theorem \ref{B-O} with $\widetilde{f}(z)$. Then the
purpose of the remaining part of the proof is to show the
following: if $\gcd(p,4N)=1$, an odd prime $\ell$ divides some
$n_i$, and
\begin{equation}
\theta^{p-1}(f(z))\equiv \sum_{i=1}^t\sum_{m=0}^{\infty}a(n_i
m^2)q^{n_i m^2} \pmod{p},\end{equation} then
$p|(\ell-1)\ell(\ell+1)N$ or $\ell \mid N$. We assume that there
exists a prime $\ell_1$ such that $\ell_1 | n_1$, $p \nmid
(\ell_1-1)\ell_1(\ell_1+1)N$ and $\ell \mid N$. We also assume
that $n_t=1$ and that $n_i \nmid n_1$ for every $i$, $2\leq i\leq
t-1$. Then, we can take a prime $\ell_i$ for each $i$, $2 \leq i
\leq t-1$, such that $\ell_i|n_i$ and $\ell_i \nmid n_1$. For
convention, we define
$$\left(\frac{n}{2}\right):=\left\{\begin{array}{cl}
  (-1)^{(n-1)^2/8} \quad&\text{if } n \text{ is odd,}\\
  0                \quad&\text{otherwise},
\end{array} \right.$$ and $\chi_{Q}(d):=\left(\frac{d}{Q} \right)$
for a prime $Q$. Let $\psi(d):=\prod_{i=2}^{t-1}\chi_{\ell_i}(d).$
We take a prime $\beta$ such that $\psi(n_1)\chi_{\beta}(n_1)=-1$.
If we denote the $\psi$-twist of $\widetilde{f}(z)$ by
$\widetilde{f}_{\psi}(z)$ and the $\psi\chi_{\beta}$-twist of
$\widetilde{f}(z)$ by $\widetilde{f}_{\psi\chi_{\beta}}(z)$, then
$$\widetilde{f}_{\psi\chi_{\beta}^2}(z)-\widetilde{f}_{\psi\chi_{\beta}}(z) \equiv
2\sum_{\gcd(m,\beta\prod_{j\geq2}\ell_{j})=1}a(n_1m^2)q^{n_1m^2}
\pmod{p}$$ and $\widetilde{f}_{\psi\chi_{\beta}}(z) \in
S_{\lambda+(p+1)(p-1)+\frac{1}{2}}(\Gamma_0(N\alpha^2\beta^2),\chi)\cap\mathbb{Z}[[q]]$
(see Chapter 3 in \cite{O}). Note that
$$\gcd(N\alpha^2\beta^2,p)=\gcd(N\alpha^2\beta^2,\ell_1)=1.$$ Thus,
$(\widetilde{f}_{\psi}(z)-\widetilde{f}_{\psi\chi_{\beta}}(z))|T(\ell_1^2,\lambda+(p+1)(p-1),\chi)$
satisfies the formula (\ref{hecke}) of Theorem \ref{B-O} for  both
of $\epsilon=1$ and $\epsilon=-1$. This results in a contradiction
since
$$(\widetilde{f}_{\psi}(z)-\widetilde{f}_{\psi\chi_{\beta}}(z))|T(\ell_1^2,\lambda+(p+1)(p-1),\chi)
\not \equiv  0 \pmod{p}$$ and $p\geq5$. Thus, we complete the
proof.
\end{proof}

\section{\bf Proofs of Theorem \ref{main2}, \ref{main3}, and \ref{main4}}
\subsection{Proof of Theorem \ref{main2}}

Note that $h(z)=\frac{\eta(z)^2}{\eta(2z)}\cdot
\frac{E_4(4z)}{\eta(4z)^6}$ is a meromorphic modular form. In
\cite{A-O} it was obtained a holomorphic modular form on
$\Gamma_0(4p^2)$ whose Fourier coefficients generate traces of
singular moduli modulo $p$ (see the formula (\ref{h-p}) and
(\ref{3.2})). Since the level of this modular form is not
relatively prime to $p$, we need the following proposition.
\begin{prop}[\cite{A-B}]\label{reduction}
Suppose that $p \geq 5$ is a prime. Also, suppose that $p \nmid
N$, $j\geq 1$ is an integer, and
$$g(z)=\sum_{n=1}^{\infty}a(n)q^n \in
S_{\lambda+\frac{1}{2}}(\Gamma_0(Np^j))\cap\mathbb{Z}[[q]].$$
Then, there exists a cusp form $G(z)\in
S_{\lambda'+\frac{1}{2}}(\Gamma_0(N))\cap\mathbb{Z}[[q]]$ such
that $$G(z)\equiv g(z) \pmod{p},$$ where
$\lambda'+\frac{1}{2}=(\lambda+\frac{1}{2})p^{j}+p^e(p-1)$ for a
sufficiently large $e\in \mathbb{N}$.
\end{prop}
\noindent Using Theorem \ref{main1} and Proposition
\ref{reduction}, we give the proof of Theorem \ref{main2}.
\begin{proof}[Proof of Theorem \ref{main2}]
Let
\begin{equation}\label{h-p}
h_{1,p}(z):=h(z)-\left(\frac{-1}{p}\right)h_{\chi_p}(z),
\end{equation}
where $h_{\chi_p}(z)$ is the $\chi_p$-twist of $h(z)$. From
(\ref{Za}), we have
$$h_{1,p}(z):=-2-\sum_{\begin{smallmatrix}
  0<d\equiv0,3\pmod{4} \\
  p|d
\end{smallmatrix}}t_1(d)q^d-2\sum_{\begin{smallmatrix}
  0<d\equiv0,3\pmod{4} \\
  \left(\frac{-d}{p}\right)=-1
\end{smallmatrix}}t_1(d)q^d$$ and
\begin{align*}h_{m,p}(z)&:=h_{1,p}(z)|T(m^2,1,\chi_0)\\
&=-2-\sum_{\begin{smallmatrix}
  0<d\equiv0,3\pmod{4} \\
  p|d
\end{smallmatrix}}t_m(d)q^d-2\sum_{\begin{smallmatrix}
  0<d\equiv0,3\pmod{4} \\
  \left(\frac{-d}{p}\right)=-1
\end{smallmatrix}}t_m(d)q^d
\end{align*}
for every positive integer $m$. Let
$$F_p(z):=\frac{\eta(4z)^{p^2}}{\eta(4pz)}.$$ It was proved in
\cite{A-O} that if $\alpha$ is a sufficiently large positive
integer, then $h_{1,p}(z)F_p(z)^{\alpha} \in
M_{\frac{3}{2}+k_0}(\Gamma_0(4p^2))$ and
\begin{equation}\label{3.2}
h_{1,p}(z)F_p(z)^{\alpha} \equiv h_{1,p}(z) \pmod{p},
\end{equation}
where $k_0=\alpha\cdot\frac{p^2-1}{2}$. Lemma \ref{l-main1} and
Proposition \ref{reduction} imply that there exists $f_{1,p}(z)\in
S_{\lambda'+\frac{1}{2}}(\Gamma_0(4))\cap\mathbb{Z}[[q]]$ such
that $$f_{1,p}(z)\equiv -2\sum_{\begin{smallmatrix}
  0<d\equiv0,3\pmod{4} \\
  \left(\frac{-d}{p}\right)=-1
\end{smallmatrix}}t_m(d)q^d\pmod{p},$$
where $\lambda'=(k_0+1+(p-1)(p+1)+\frac{1}{2})p^2+p^e(p-1)$ for a
sufficiently large $e\in \mathbb{N}$.

We assume that the coefficients of $f_{1,p}(z)$ do not satisfy
Property A for an odd prime $p\equiv2 \pmod{3}$. Note that
$\left(\frac{-3}{p}\right)=-1$ and that $p \nmid (3-1)3(3+1)$. So,
Theorem \ref{main1} implies that
$$2t_1(3) \equiv0 \pmod{p}.$$ This results in a contradiction since
$2t_1(3)=2^4\cdot31$. Thus, we obtain a proof when $m=1$.

For every odd prime $\ell$, we have
\begin{align*}
f_{1,p}(z)|T(\ell^2,\lambda',\chi_0)&\equiv
\theta^{p-1}(h_{1,p}(z))|T(\ell^2,\lambda',\chi_0) \\
&\equiv \theta^{p-1}(h_{1,p}(z)|T(\ell^2,1,\chi_0)) \equiv
\theta^{p-1}(h_{\ell,p}(z)) \pmod{p}.
\end{align*}
 Moreover, Lemma \ref{2-3} implies
that a positive proportion of the primes $\ell$ satisfies the
property
$$f_{1,p}(z)|T(\ell^2,\lambda',\chi_0)\equiv 2f_{1,p}
\pmod{p}.$$ This completes the proof.
\end{proof}
\subsection{Proofs of Theorem \ref{main3}}
The following theorem gives the formula for the Hurwitz class
number in terms of the Fourier coefficients of a modular form of
half integral weight.
\begin{thr}\label{H-C} Let $T(z):=1+2\sum_{n=1}^{\infty}q^{n^2}$.
If integers $r_3(n)$ are defined as
$$\sum_{n=0}^{\infty}r_3(n)q^n:=T(z)^3,$$ then $$r(n)=\left\{\begin{array}{ll}
  12H(-4n) &\text{ if } n\equiv 1,2 \pmod{4}, \\
  24H(-n) &\text{ if } n\equiv 3 \pmod{8}, \\
  r(n/4)   &\text{ if } n\equiv 0 \pmod{4}, \\
  0        &\text{ if } n\equiv 7 \pmod{8}.
\end{array}\right.$$
\end{thr}

Note that $T(z)$ is a half integral weight modular form of weight
$\frac{1}{2}$ on $\Gamma_0(4)$. Combining Theorem \ref{main1} and
Theorem \ref{H-C}, we derive the proof of Theorem \ref{main3}.
\begin{proof}[Proof of Theorem \ref{main3}]
Let $G(z)$ be the $\left(\frac{4}{}\right)$-twist of $T(z)^3$.
Then, from Theorem \ref{H-C}, we have
\begin{equation*}
G(z)=1+\sum_{n\equiv 1 \pmod{4}}12H(-4n)q^n+\sum_{n\equiv 3
\pmod{8}}24H(-n)q^n
\end{equation*}
 and $G(z)\in M_{\frac{3}{2}}(\Gamma_0(16))$. Note
that $24H(-3)=8$. This gives the complete proof by Theorem
\ref{main1}.
\end{proof}

\subsection{Proofs of Theorem \ref{main4}}In the following, we prove Theorem \ref{main4}.
\begin{proof}[Proof of  Theorem \ref{main4}]
Let $$W(z):=\frac{\eta(2z)}{\eta(z)^2}.$$ It is known that
$$W(z)=\sum_{n=0}^{\infty}\bar{P}(n)q^n$$ and that $W(z)$ is a weakly holomorphic modular form
on $\Gamma_0(16)$. Let
$$G(z):=\left(W(z)-\left(\frac{-1}{p}\right)W_{\chi_p}(z)\right)F_p(z)^{p^{\beta}},$$ where
$F_p(z)=\frac{\eta(4z)^{p^2}}{\eta(4p^2 z)}$ and $\beta$ are
positive integers.
Then we have $$G(z)\equiv 2\sum_{%
\begin{smallmatrix}
  0<n \\
  \left(\frac{-n}{p}\right)=-1
\end{smallmatrix}%
}\bar{P}(n)q^n+\sum_{%
\begin{smallmatrix}
  0<n \\
  p|n
\end{smallmatrix}%
}\bar{P}(n)q^n \pmod{p}. $$

We claim that there exists a positive integer $\beta$ such that
$G(z)$ is a holomorphic modular form of half integral weight on
$\Gamma_0(16p^2)$. To prove our claim, we follow the arguments of
Ahlgren and Ono (\cite{A-B}, Lemma 4.2). Note that, by a
well-known criterion, $F_p(z)$ is a holomorphic modular form on
$\Gamma_0(4p^2)$ that vanishes at each cusp $\frac{a}{c} \in
\mathbb{Q}$ for which $p^2 \nmid c$ (see \cite{G-H}). This implies
that $G(z)$ is a weakly holomorphic modular form on
$\Gamma_0(16p^2)$. If $\beta$ is sufficiently large, then $G(z)$
is holomorphic except at each cusp $\frac{a'}{c'}$ for which
$p^2|c'$.

Thus, we prove that $G(z)$ is holomorphic at $\frac{1}{2^mp^2}$
for $0 \leq m \leq 3$. Let, for odd $d$,
$$\epsilon_d:=\left\{\begin{array}{c}
  1 \text{ if } d\equiv 1 \pmod{4}, \\
  i \text{ if } d\equiv 3 \pmod{4}.
\end{array}\right.$$ If $f(z)$ is a function on the complex upper
half plane, $\lambda\in
\mathbb{Z}$, and $\gamma=\left(%
\begin{smallmatrix}
  a & b \\
  c & d
\end{smallmatrix}%
\right)\in \Gamma_0(4)$, then we define the usual slash operator
by
\begin{equation*}
f(z) \mid_{\lambda+\frac{1}{2}}
\gamma:=\left(\frac{c}{d}\right)^{2\lambda+1}
\epsilon_d^{-1-2\lambda}(cz+d)^{-\lambda-\frac{1}{2}}f\left(\frac{az+b}{cz+d}\right).
\end{equation*} Let
$g:=\sum_{v=1}^{\infty}\left(\frac{v}{p}\right)e^{2 \pi i v/p}$ be
the usual Gauss sum. Note that
$$W_{\chi_p}(z)=\frac{g}{p}\sum_{v=1}^{p-1}\left(\frac{v}{p}\right)W(z)|_{-\frac{1}{2}}
\left(%
\begin{smallmatrix}
  1 & -v/p \\
  0 & 1 \\
\end{smallmatrix}%
\right).$$ Choose an integer $k_v$ satisfying $$16k_v \equiv 15v
\pmod{p}.$$ Then, we have \begin{equation}\label{3.3}\left(%
\begin{matrix}
  1 & -\frac{v}{p} \\
  0 & 1 \\
\end{matrix}%
\right)\left(%
\begin{matrix}
  1 & 0 \\
  2^m p^2 & 1 \\
\end{matrix}%
\right)=\gamma_{v,m}\left(%
\begin{matrix}
  1 & 0 \\
  2^mp^2 & 1 \\
\end{matrix}%
\right)\left(%
\begin{matrix}
  1 & -\frac{16v}{p}+\frac{16k_v}{p} \\
  0 & 1 \\
\end{matrix}%
\right),\end{equation} where $$\gamma_{v,m}=\left(%
\begin{matrix}
  1-2^{m+4}p(v+k_v+2^{m}v^2p-2^{m}vk_vp) & \frac{1}{p}(15v-16k_v-2^{m+4}(v^2p+vk_vp)) \\
 2^{2m} p^2(-16vp+16k_v p) & 2^{m+4}vp-2^{m+4}k_v p+1 \\
\end{matrix}%
\right).$$ Note that $W(z)$ has its only pole at $z \sim 0$ up to
$\Gamma_0(16)$. Since $\gamma_{v,m} \in \Gamma_0(16)$, the formula
(\ref{3.3}) implies that $W_{\chi_p}(z)$ is holomorphic at
$2^mp^2$ for $1\leq m \leq 3$. Thus, $G(z)$ is holomorphic at
$2^mp^2$ for $1\leq m \leq 3$.

If $m=0$, then we have
$$W(z)|_{-\frac{1}{2}}\gamma_{v,0}=\left(\frac{-16vp^3+16k_v p^3}{16vp-16k_v
p+1}\right)W(z)=\left(\frac{p^2(-vp+k_v p)}{16vp-16k_v
p+1}\right)W(z)=W(z).$$ Note that \begin{equation}\label{3.4}
W(z)|_{-\frac{1}{2}}\left(%
\begin{smallmatrix}
  1 & 0 \\
  p^2 & 1 \\
\end{smallmatrix}%
\right)=\alpha \cdot q^{-\frac{1}{16}}+O(1),
\end{equation}
where $\alpha$ is a nonzero complex number.  The $q$-expansion of
$W_{\chi_p}(z)$ at $\frac{1}{p^2}$ is given by
\begin{equation}\label{3.5}
W_{\chi_p}(z)|_{-\frac{1}{2}}\left(%
\begin{smallmatrix}
  1 & 0 \\
  p^2 & 1 \\
\end{smallmatrix}%
\right).
\end{equation}
Using (\ref{3.3}) and (\ref{3.4}), the only term in (\ref{3.5})
with a negative exponent on $q$ is the term
$$\frac{g}{p}\alpha q^{-\frac{1}{16}}\sum_{v=1}^{p-1}\left(\frac{v}{p}\right)e^{\frac{2\pi i}{p}(v-k_v)}.$$
If N is defined by $16N \equiv 1 \pmod{p}$, then we have
$$\frac{g}{p}\alpha q^{-\frac{1}{16}}\sum_{v=1}^{p-1}\left(\frac{v}{p}\right)e^{\frac{2\pi i}{p}(v-k_v)}
=\frac{g}{p}\alpha
q^{-\frac{1}{16}}\sum_{v=1}^{p-1}\left(\frac{v}{p}\right)e^{\frac{2\pi
i}{p}Nv}=\frac{g^2}{p}\alpha q^{-\frac{1}{16}}
=\left(\frac{-1}{p}\right)\alpha q^{-\frac{1}{16}}.$$ Thus, we
have that
$$(W(z)-W_{\chi_p}(z))|_{-\frac{1}{2}}\left(%
\begin{smallmatrix}
  1 & 0 \\
  p^2 & 1 \\
\end{smallmatrix}\right)=O(1).$$
This implies that $G(z)$ is a holomorphic modular form of half
integral weight on $\Gamma_0(16p^2)$. Noting that
$$\bar{P}(3)=8,$$ the remaining part of the proof is similar to
that in Theorem \ref{main3}. Thus, it is omitted.
\end{proof}

\end{document}